\documentclass{amsart}
\usepackage{amsfonts}
\begin{document}
\newtheorem{defn}{Definition}[section]
\newtheorem{thm}{Theorem}[section]
\newtheorem{prop}{Proposition}[section]
\newtheorem{exam}{Example}[section]
\newtheorem{cor}{Corollary}[section]
\newtheorem{rem}{Remark}[section]
\newtheorem{lem}{Lemma}[section]
\newcommand{\CC}{\mathbb{C}}
\newcommand{\KK}{\mathbb{K}}
\newcommand{\ZZ}{\mathbb{Z}}
\def\a{{\alpha}}
\def\b{{\beta}}
\def\d{{\delta}}
\def\g{{\gamma}}
\def\l{{\lambda}}
\def\gg{{\mathfrak g}}
\def\cal{\mathcal }

\title{ On the description of Leibniz superalgebras of nilindex $n+m$}
\author{L.M. Camacho, J. R. G\'{o}mez, A.Kh. Khudoyberdiyev and B.A. Omirov}
\address{[L.M. Camacho -- J.R. G\'{o}mez] Dpto. Matem\'{a}tica Aplicada I.
Universidad de Sevilla. Avda. Reina Mercedes, s/n. 41012 Sevilla.
(Spain)} \email{lcamacho@us.es --- jrgomez@us.es}
\address{[A.Kh. Khudoyberdiyev -- B.A. Omirov] Institute of Mathematics and
Information Technologies of Academy of Uzbekistan, 29, F.Hodjaev
srt., 100125, Tashkent (Uzbekistan)} \email{khabror@mail.ru ---
omirovb@mail.ru}

\thanks{The last author was supported by grant NATO-Reintegration ref.
CBP.EAP.RIG.983169}

\maketitle
\begin{abstract}
In this work we investigate the complex Leibniz superalgebras with
characteristic sequence $(n_1,\dots,n_k|m)$ and nilindex $n+m$,
where $n=n_1+\cdots+n_k,$ $n$ and $m$ ($m\neq 0$) are dimensions
of even and odd parts, respectively. Such superalgebras with
condition $n_1 \geq n-1$ were classified in
\cite{FilSup}--\cite{C-G-O-Kh}. Here we prove that in the case
$n_1 \leq n-2$ the Leibniz superalgebras have nilindex less than
$n+m.$ Thus, we get the classification of Leibniz superalgebras
with characteristic sequence $(n_1,\dots,n_k|m)$ and nilindex
$n+m.$
\end{abstract} \maketitle

\textbf{Mathematics Subject Classification 2000}: 17A32, 17B30,
17B70, 17A70.

\textbf{Key Words and Phrases}: Lie superalgebras, Leibniz
superalgebras, nilindex, characteristic sequence, naturally
gradation.

\section{Introduction}

The paper is devoted to the study of the nilpotent Leibniz
superalgebras. The notion of Leibniz superalgebras appears
comparatively recently \cite{Alb}, \cite{Liv}. Leibniz
superalgebras are generalizations of Leibniz algebras \cite{Lod}
and, on the other hand, they naturally generalize well-known Lie
superalgebras. The elementary properties of Leibniz superalgebras
were obtained in \cite{Alb}.

In the work \cite{G-K-N} the Lie superalgebras with maximal
nilindex were classified. The distinguishing property of such
superalgebras is that they are two-generated and its nilindex
equal to $n+m$ (where $n$ and $m$ are dimensions of even and odd
parts, respectively). In fact, there exists unique Lie
superalgebra of the maximal nilindex and its characteristic
sequence is equal to $(1,1 |m).$ This superalgebra is filiform Lie
superalgebra (the characteristic sequence equal to $(n-1,1 | m)$)
and we mention about paper \cite{2007Yu}, where some crucial
properties of filiform Lie superalgebras are given.

In case of Leibniz superalgebras the property of maximal nilindex
is equivalent to the property of single-generated superalgebras
and they are described in \cite{Alb}. However, the description of
Leibniz superalgebras of nilindex $n+m$ is an open problem and it
needs to solve many technical tasks. Therefore, they can be
studied by applying restrictions on their characteristic
sequences. In the present paper we consider Leibniz superalgebras
with characteristic sequence $(n_1, \dots, n_k | m)$ and nilindex
$n+m.$ Since such superalgebras in the case $n_1 \geq n-1$ have
been already classified in \cite{FilSup}-\cite{C-G-O-Kh} we need
to study the case $n_1 \leq n-2.$

In fact, in the previous cases (cases where $n_1 \geq n-1$) due to
work \cite{AO1} we have used some information on the structure of
even part of the superalgebra and it played the crucial role in
that classifications. In the rest case (case where $n_1 \leq n-2$)
the structure of even part is unknown, but we used the properties
of natural gradation and the naturally graded basis (so-called
adapted basis) of even part of the superalgebra.

All over the work we consider spaces over the field of complex
numbers. By asterisks $(*)$ we denote the appropriate coefficients
at the basic elements of superalgebra.

\section{Preliminaries}
Recall the notion of Leibniz superalgebras.

\begin{defn} A $\mathbb{Z}_2$-graded vector space $L=L_0\oplus
L_1$ is called a Leibniz superalgebra if it is equipped with a
product $[-, -]$ which satisfies the following conditions:

1. $[L_\alpha,L_\beta]\subseteq L_{\alpha+\beta(mod\ 2)},$

2. $[x, [y, z]]=[[x, y], z] - (-1)^{\alpha\beta} [[x, z], y]-$
  Leibniz superidentity,\\
 for all $x\in L,$ $y \in L_\alpha,$ $z \in
L_\beta$ and $\alpha,\beta\in \mathbb{Z}_2.$
\end{defn}

Evidently, even part of the Leibniz superalgebra is a Leibniz
algebra.

The vector spaces $L_0$ and $L_1$ are said to be even and odd
parts of the superalgebra $L$, respectively.

Note that if in Leibniz superalgebra $L$ the identity
$$[x,y]=-(-1)^{\alpha\beta} [y,x]$$ holds for any $x \in
L_{\alpha}$ and $y \in L_{\beta},$ then the Leibniz superidentity
can be transformed into the Jacobi superidentity. Thus, Leibniz
superalgebras are a generalization of Lie superalgebras and
Leibniz algebras.

The set of all Leibniz superalgebras with the dimensions of the
even and odd parts, respectively equal to $n$ and $m$, we denote
by $Leib_{n,m}.$

For a given Leibniz superalgebra $L$ we define the descending
central sequence as follows:
$$
L^1=L,\quad L^{k+1}=[L^k,L], \quad k \geq 1.
$$

\begin{defn} A Leibniz superalgebra $L$ is called
nilpotent, if there exists  $s\in\mathbb N$ such that $L^s=0.$ The
minimal number $s$ with this property is called nilindex of the
superalgebra $L.$
\end{defn}

The following theorem describes nilpotent Leibniz superalgebras with
maximal nilindex.
\begin{thm} \label{t1} \cite{Alb} Let $L$  be a Leibniz superalgebra
of $Leib_{n,m }$ with nilindex equal to $n+m+1.$ Then $L$ is
isomorphic to one of the following non-isomorphic superalgebras:
$$
[e_i,e_1]=e_{i+1},\ 1\le i\le n-1, \ m=0;\quad \left\{ \begin{array}{ll} [e_i,e_1]=e_{i+1},& 1\le i\le n+m-1, \\
{[}e_i,e_2{]}=2e_{i+2}, & 1\le i\le n+m-2,\\ \end{array}\right.
$$
(omitted products are equal to zero).
\end{thm}

\begin{rem} {\em
From the assertion of Theorem \ref{t1} we have that in case of
non-trivial odd part $L_1$ of the superalgebra $L$ there are two
possibility for $n$ and $m$, namely, $m=n$ if $n+m$ is even and
$m=n+1$ if $n+m$ is odd. Moreover, it is clear that the Leibniz
superalgebra has the maximal nilindex if and only if it is
single-generated.}
\end{rem}

Let $L=L_0\oplus L_1$ be a nilpotent Leibniz superalgebra. For an
arbitrary element $x\in L_0,$ the operator of right multiplication
$R_x$ defined as $R_x(y)=[y,x]$ is a nilpotent endomorphism of the
space $L_i,$ where $i\in \{0, 1\}.$ Taking into account the
property of complex field we can consider the Jordan form for
$R_x.$ For operator $R_x$ denote by $C_i(x)$ ($i\in \{0, 1\}$) the
descending sequence of its Jordan blocks dimensions. Consider the
lexicographical order on the set $C_i(L_0)$.

\begin{defn} \label{d4}A sequence
$$C(L)=\left( \left.\max\limits_{x\in L_0\setminus L_0^2} C_0(x)\
\right|\ \max\limits_{\widetilde x\in L_0\setminus L_0^2}
C_1\left(\widetilde x\right) \right) $$ is said to be the
characteristic sequence of the Leibniz superalgebra $L.$
\end{defn}

Similarly to \cite{GL} (corollary 3.0.1) it can be proved that the
characteristic sequence is invariant under isomorphism.

Since Leibniz superalgebras from $Leib_{n,m}$ with characteristic
sequence equal to $(n_1, \dots, n_k| m)$ and nilindex $n+m$ are
already classified for the case $n_1 \geq n-1,$ henceforth we
shall reduce our investigation to the case of $n_1 \leq n-2.$

From the Definition \ref{d4} we conclude that a Leibniz algebra
$L_0$ has characteristic sequence $(n_1, \dots, n_k).$ Let $s \in
\mathbb{N}$ be a nilindex of the Leibniz algebra $L_0.$ Since $n_1
\leq n-2$ then we have $s \leq n-1$ and Leibniz algebra $L_0$ has
at least two generators (the elements which belong to the set
$L_0\setminus L_0^2$).

For the completeness of the statement below we present the
classifications of the papers \cite{FilSup}, \cite{C-G-O-Kh} and
\cite{G-K-N}.

$Leib_{1,m}:$
$$\small
\left\{\begin{array}{l} [y_i,x_1]=y_{i+1}, \ \ 1\leq i \leq m-1,
\end{array}\right.\\
\\$$

\

$Leib_{n,1}:$

$$ \small\left\{\begin{array}{ll} [x_i,x_1]=x_{i+1},& 1 \leq i \leq n-1,\\{}
[y_1,y_1]=\alpha x_n, & \alpha = \{0, \ 1\},\end{array}\right.$$

\

$Leib_{2,2}:$
$$\small\begin{array}{ll}
\left\{\begin{array}{l} [y_1,x_1]=y_2, \\ {[}x_1,y_1]=\displaystyle \frac12 y_2, \\[2mm] {[}x_2,y_1]=\displaystyle
y_2, \\[2mm] [y_1,x_2] = 2y_2, \\ {[}y_1,y_1]=x_2, \\
\end{array}\right.&
 \left\{\begin{array}{l}
[y_1,x_1]=y_2, \\ {[}x_2,y_1]=\displaystyle y_2, \\[2mm] {[}y_1,x_2]= 2y_2, \\ {[}y_1,y_1]=x_2, \\
\end{array}\right.
\end{array}$$

\

$Leib_{2,m}:$
$$\small\begin{array}{ll}
\left\{\begin{array}{ll} [x_1,x_1]=x_2, & m\geq 3\\{}
[y_i,x_1]=y_{i+1},& 1\leq i\leq m-1,\\{} [x_1,y_i]=-y_{i+1},&1\leq
i\leq m-1,\\{} [y_i,y_{m+1-i}]=(-1)^{i+1}x_2, & 1\leq i\leq m-1.
\end{array}\right.&\left\{\begin{array}{ll}
[y_i,x_1]=-[x_1,y_i]=y_{i+1},& 1\le i\le m-1,
\\[2mm] [y_{m+1},y_i]=(-1)^{i+1}x_2,& 1\le i\le \frac{m+1}2,
\end{array}\right.
\end{array}$$

\

In order to present the classification of Leibniz superalgebras
with characteristic sequence $(n-1,1 | m)$, $n \geq 3$ and
nilindex $n+m$ we need to introduce the following families of
superalgebras:

$$\bf Leib_{n,n-1}:$$
$L(\alpha_4, \alpha_5, \ldots, \alpha_n, \theta):$
$$
\left\{\begin{array}{ll}
[x_1,x_1]=x_3,& \\[1mm] {[}x_i,x_1]=x_{i+1},&    2 \le i \le n-1,
\\[1mm]
{[}y_j,x_1]=y_{j+1},&    1 \le j \le n-2,
\\[1mm]
{[}x_1,y_1]= \frac12 y_2,&
\\[1mm]
{[}x_i,y_1]= \frac12 y_i,  &    2 \le i \le n-1,
\\[1mm]
{[}y_1,y_1]=x_1,&
\\[1mm]
{[}y_j,y_1]=x_{j+1},& 2 \le j \le n-1,
\\[1mm]
{[}x_1,x_2]=\alpha_4x_4+ \alpha_5x_5+ \ldots +
\alpha_{n-1}x_{n-1}+ \theta x_n,&
\\[1mm]
{[}x_j,x_2]= \alpha_4x_{j+2}+ \alpha_5x_{j+3}+ \ldots +
\alpha_{n+2-j}x_n,& 2 \le j \le n-2,
\\[1mm]
{[}y_1,x_2]= \alpha_4y_3+ \alpha_5y_4+ \ldots +
\alpha_{n-1}y_{n-2}+\theta y_{n-1},&
\\[1mm]
{[}y_j,x_2]= \alpha_4y_{j+2}+ \alpha_5y_{j+3}+ \ldots +
\alpha_{n+1-j}y_{n-1},& 2 \le j \le n-3. \end{array} \right.$$

$G(\beta_4,\beta_5, \ldots, \beta_n, \gamma):$
$$ \left\{\begin{array}{ll}
[x_1,x_1]=x_3, \\[1mm] {[}x_i,x_1]=x_{i+1},&    3 \le i \le n-1,
\\[1mm]
{[}y_j,x_1]=y_{j+1}, &    1 \le j \le n-2,
\\[1mm]
{[}x_1,x_2]= \beta_4x_4+\beta_5x_5+\ldots+\beta_nx_n,&
\\[1mm]
{[}x_2,x_2]= \gamma x_n,& \\[1mm]
{[}x_j,x_2]= \beta_4x_{j+2}+\beta_5x_{j+3}+\ldots+\beta_{n+2-j}x_n,& 3\le j\le n-2, \\[1mm]
{[}y_1,y_1]=x_1,&
\\[1mm]
{[}y_j,y_1]=x_{j+1},& 2 \le j \le n-1,
\\[1mm]
{[}x_1,y_1]= \frac12 y_2,&
\\[1mm]
{[}x_i,y_1]= \frac12 y_i,& 3\le i\le n-1,
\\[1mm]
{[}y_j,x_2]= \beta_4y_{j+2}+\beta_5y_{j+3}+ \ldots +
\beta_{n+1-j}y_{n-1},& 1\le j\le n-3. \end{array} \right.$$
$$\bf Leib_{n,n}:$$
$M(\alpha_4, \alpha_5, \ldots, \alpha_n, \theta, \tau):$
$$ \left\{
\begin{array}{ll}
[x_1,x_1]=x_3,& \\[1mm] {[}x_i,x_1]=x_{i+1},&     2 \le i \le n-1,
\\[1mm]
{[}y_j,x_1]=y_{j+1}, &    1 \le j \le n-1,
\\[1mm]
{[}x_1,y_1]=  \frac12 y_2,&
\\[1mm]
{[}x_i,y_1]=  \frac12 y_i, &    2 \le i \le n,
\\[1mm]
{[}y_1,y_1]=x_1,&
\\[1mm]
{[}y_j,y_1]=x_{j+1},& 2 \le j \le n-1,
\\[1mm]
{[}x_1,x_2]=\alpha_4x_4+ \alpha_5x_5+ \ldots +
\alpha_{n-1}x_{n-1}+ \theta x_n,&
\\[1mm]
{[}x_2,x_2]=\gamma_4x_4,&\\[1mm]
 {[}x_j,x_2]= \alpha_4x_{j+2}+
\alpha_5x_{j+3}+ \ldots + \alpha_{n+2-j}x_n,&3 \le j \le n-2,
\\[1mm]
{[}y_1,x_2]= \alpha_4y_3+ \alpha_5y_4+ \ldots +
\alpha_{n-1}y_{n-2}+\theta y_{n-1}+\tau y_n,&
\\[1mm]
{[}y_2,x_2]= \alpha_4y_4+ \alpha_5y_4+ \ldots +
\alpha_{n-1}y_{n-1}+\theta y_n,&
\\[1mm]
{[}y_j,x_2]= \alpha_4y_{j+2}+ \alpha_5y_{j+3}+ \ldots +
\alpha_{n+2-j}y_{n},& 3 \le j \le n-2.\end{array} \right.$$

$H(\beta_4, \beta_5, \ldots,\beta_n, \delta , \gamma ):$
$$ \left\{
\begin{array}{ll}
[x_1,x_1]=x_3,& \\[1mm] {[}x_i,x_1]=x_{i+1},&     3 \le i \le n-1,
\\[1mm]
{[}y_j,x_1]=y_{j+1}, &    1 \le j \le n-2,
\\[1mm]
{[}x_1,x_2]= \beta_4x_4+\beta_5x_5+\ldots+\beta_nx_n,&
\\[1mm]
{[}x_2,x_2]= \gamma x_n, &\\[1mm]
{[}x_j,x_2]= \beta_4x_{j+2}+\beta_5x_{j+3}+\ldots+\beta_{n+2-j}x_n,& 3\le j\le n-2, \\[1mm]
{[}y_1,y_1]=x_1,&
\\[1mm]
{[}y_j,y_1]=x_{j+1},& 2 \le j \le n-1,
\\[1mm]
{[}x_1,y_1]= \frac12 y_2,&
\\[1mm]
{[}x_i,y_1]= \frac12 y_i,& 3\le i\le n-1,
\\[1mm]
{[}y_1,x_2]= \beta_4y_3+\beta_5y_4+ \ldots + \beta_ny_{n-1}+\delta y_n,& \\[1mm]
{[}y_j,x_2]= \beta_4y_{j+2}+\beta_5y_{j+3}+ \ldots +
\beta_{n+2-j}y_n,& 2\le j\le n-2. \end{array} \right.$$

Let us introduce also the following operators which act on
$k$-dimensional vectors:
$$
V^m_{j,k}(\alpha_1,\alpha_2,\ldots,\alpha_k)=( 0, 0,\ldots,
\stackrel{j-1}{0} ,1, S_{m,j}^{j+1}\alpha_{j+1},
S_{m,j}^{j+2}\alpha_{j+2},\ldots S_{m,j}^{k-1}\alpha_{k-1},
S_{m,j}^k\alpha_k),
$$ $$
V^m_{k+1,k}(\alpha_1,\alpha_2,\ldots,\alpha_k)=( 0, 0,\ldots, 0),
$$ $$
W^m_{s,k}(0,0,\ldots,\stackrel{j-1}{0},\stackrel{j}{1},S_{m,j}^{j+1}\alpha_{j+1},S_{m,j}^{j+2}\alpha_{j+2},\ldots,
S_{m,j}^k\alpha_k,\gamma)=
$$ $$
=( 0, 0,\ldots, \stackrel{j}{1} ,0,\ldots,\stackrel{s+j}{1},
S_{m,s}^{s+1}\alpha_{s+j+1}, S_{m,s}^{s+2}\alpha_{s+j+2},\ldots,
S_{m,s}^{k-j}\alpha_k, S_{m,s}^{k+6-2j}\gamma),
$$  $$
W^m_{k+1-j,k}(0,0,\ldots,\stackrel{j-1}{0},\stackrel{j}{1},S_{m,j}^{j+1}\alpha_{j+1},S_{m,j}^{j+2}\alpha_{j+2},\ldots,
S_{m,j}^k\alpha_k,\gamma)=$$ $\qquad
=(0,0,\ldots,\stackrel{j}{1},0,\ldots,1), $ $$
W^m_{k+2-j,k}(0,0,\ldots,\stackrel{j-1}{0},\stackrel{j}{1},S_{m,j}^{j+1}\alpha_{j+1},S_{m,j}^{j+2}\alpha_{j+2},\ldots,
S_{m,j}^k\alpha_k,\gamma)=$$ $\qquad
=(0,0,\ldots,\stackrel{j}{1},0,\ldots,0), $

\noindent where $k\in N,$ $1\le j\le k,$ $1\le s\le k-j,$
$\displaystyle S_{m,t}=\cos\frac{2\pi m}t+i\sin\frac{2\pi m}t$
$(m=0,1,\ldots, t-1).$

Below we present the complete list of pairwise non-isomorphic
Leibniz superalgebras with characteristic sequence equal to
$(n-1,1 | m)$  and nilindex $n+m:$
$$
\begin{array}{l} L\left( V_{j,n-3}\left( \alpha_4,\alpha_5,\ldots,
\alpha_n\right),S_{m,j}^{n-3}\theta\right),\qquad \ \
1\le j\le n-3, \\[2mm] L(0,0,\ldots,0,1), \ L(0,0,\ldots,0), \ G(0,0,\ldots,0,1), \
G(0,0,\ldots,0),
\\[2mm] G\left( W_{s,n-2}\left( V_{j,n-3}\left(
\beta_4,\beta_5,\ldots,\beta_n\right),\gamma\right)\right),\quad
1\le j\le n-3,\ 1\le s\le n-j,
\\[2mm] M\left( V_{j,n-2}\left( \alpha_4,\alpha_5,\ldots,\alpha_n\right),S_{m,j}^{n-3}\theta\right),
\qquad \ 1\le j\le n-2,
\\[2mm] M(0,0,\ldots,0,1), \ M(0,0,\ldots,0), \ H(0,0,\ldots,0,1), \
H(0,0,\ldots,0) \\[2mm] H\left( W_{s,n-1}\left( V_{j,n-2}\left(
\beta_4,\beta_5,\ldots,\beta_n\right),\gamma\right)\right),\quad
1\le j\le n-2,\ 1\le s\le n+1-j, \\\end{array}
$$
where omitted products are equal to zero.

\begin{defn} \label{d5}
For a given Leibniz algebra $A$ of the nilindex $s$ we put
$gr(A)_i = A^i / A^{i+1}, \quad 1 \leq i \leq s-1$ and $gr(A) =
gr(A)_1 \oplus gr(A)_2 \oplus \dots \oplus gr(A)_{s-1}.$ Then
$[gr(A)_i, gr(A)_j] \subseteq gr(A)_{i+j}$ and we obtain the
graded algebra $gr(A).$ The gradation constructed in this way is
called the natural gradation and if Leibniz algebra $A$ is
isomorphic to $gr(A)$ we say that the algebra $A$ is naturally
graded Leibniz algebra.
\end{defn}

Further we shall consider the basis of even part $L_0$ of the
superalgebra $L$ which corresponds with the natural gradation,
i.e. $\{x_1, \dots, x_{t_1}\}=L_0\setminus L_0^2,$ $\{x_{{t_1}+1},
\dots, x_{t_2}\}=L_0^2\setminus L_0^3, \dots, \{x_{{t_{s-2}}+1},
\dots, x_{n}\}=L_0^{s-1}.$

Since the second part of the characteristic sequence of a Leibniz
superalgebra $L$ is equal to $m$ then there exists a nilpotent
endomorphism $R_x$ ($x\in L_0\setminus L_0^2$) of the space $L_1$
such that its Jordan form consists of one Jordan block. Therefore,
we can assume the existence of an adapted basis $\{x_1, x_2,
\dots, x_n, y_1, y_2, \dots, y_n\}$ such that
$$[y_j,x_1] = y_{j+1}, \quad 1 \leq j \leq m-1. \eqno (1)$$

\section{The main result}

Let $L$ be a Leibniz superalgebra with characteristic sequence
$(n_1, \dots, n_k| m),$ $ n_1 \leq n-2$ and let $\{x_1, x_2,
\dots, x_n, y_1, y_2, \dots, y_n\}$ be the adapted basis of $L.$
In this section we shall prove that nilindex of such superalgebra
is less than $n+m.$ According to the Theorem \ref{t1} we have a
description of single-generated Leibniz superalgebras, which have
nilindex $n+m+1.$ If the number of generators is greater than two,
then evidently superalgebra has nilindex less than $n+m.$
Therefore, we should be consider case of two generators.

Note that, the case where both generators lie in even part is not
possible (since $m\neq 0$). The equality (1) implies that basic
elements $y_2, y_3, \dots, y_m$ can not to be generators.
Therefore, the first generator belongs to $L_0$ and the second one
lies in $L_1.$ Moreover, without loss of generality we can suppose
that $y_1$ is a generator. Let us find the generator of Leibniz
superalgebra $L$ which lies in $L_0.$

\begin{lem}\label{l1}
Let $L=L_0 \oplus L_1$ be a two generated Leibniz superalgebra
from $Leib_{n,m}$ with characteristic sequence equal to $(n_1,
\dots, n_k| m).$ Then $x_1$ and $y_1$ can be chosen as generators
of the $L.$
\end{lem}
\begin{proof}
As mentioned above $y_1$ can be chosen as first generator of $L$.
If $x_1\in L\setminus L^2,$ then the assertion of the lemma is
evident. If $x_1\in L^2,$ then there exists some $i_0$ ($2\leq i_0
\leq t_1$) such that $x_{i_0}\in L\setminus L^2.$ Put
$x_1'=Ax_1+x_{i_0},$ then $x_1'$ is a generator of the
superalgebra $L$ (since $x_1'\in L\setminus L^2$). Moreover,
making transformation of the basis of $L_1$ as follows
$$y'_1 = y_1, \quad y'_j = [y'_{j-1}, x'_1], \quad 2 \leq j \leq
m$$ and taking sufficiently big value of the parameter $A$ we
preserve the equality (1). Thus, in the basis $\{x_1', x_2, \dots,
x_n, y_1', y_2', \dots, y_m'\}$ of the $L$ the elements $x_1'$ and
$y_1'$ are generators.
\end{proof}

Due to Lemma \ref{l1} further we shall suppose that $\{x_1, y_1\}$
are generators of the Leibniz superalgebra $L.$

 Let us introduce the notations:

$$[x_i,y_1]= \sum\limits_{j=2}^m \alpha_{i,j}y_j,\ 1 \le i \le n, \  \ \ [y_i,y_1]=
\sum\limits_{j=2}^n \beta_{i,j}x_j, \ 1 \le i \le m. \eqno (2)$$

Since $x_1$ and $y_1$ are generators of Leibniz superalgebra $L,$
we have
$$L = \{x_1, x_2,  \dots, x_n,
y_1, y_2, \dots, y_m\},$$ $$L^2 = \{x_2, x_3,  \dots, x_n, y_2,
y_3, \dots, y_m\}.$$

If we consider the next power of $L,$ then from the multiplication
(1), obviously, we get $\{y_3, \dots, y_m\}\subseteq L^3.$
However, do not have information about the position of the element
$y_2.$

\begin{thm}\label{t2}
Let $L=L_0 \oplus L_1$  be a Leibniz superalgebra from
$Leib_{n,m}$ with characteristic sequence equal to $(n_1, \dots,
n_k | m)$ and let $y_2 \notin L^3.$ Then $L$ has a nilindex less
than $n+m.$
\end{thm}
\begin{proof}
Let us assume the contrary, i.e. nilindex of the superalgebra $L$
is equal to $n+m.$ Then the condition $y_2 \notin L^3$ deduce
$\{x_2, x_3, \dots, x_n\}\subseteq L^3.$ Therefore,
$$L^3 = \{x_2, x_3,  \dots, x_n, y_3, \dots, y_m\}.$$

Let $s\in\mathbb{N}$ be a number such that $x_2 \in L^s\setminus
L^{s+1},$ that is $$ L^s = \{x_2, x_3, \dots, x_n, y_s, \dots,
y_m\}, \  s \geq 3,$$
$$L^{s+1} = \{x_3, x_4, \dots, x_n, y_s, \dots, y_m\}.$$

It means that $x_2$ can be obtained only from product
$[y_{s-1},y_1]$ and thereby $\beta_{s-1,2} \neq 0.$

 Similarly, we assume that
$k$ is a number for which $x_3 \in L^{s+k}\setminus L^{s+k+1}.$
Then for the powers of superalgebra $L$ we have the following
$$L^{s+k} = \{x_3, x_4, \dots, x_n, y_{s+k-1}, \dots,
y_m\}, k \geq 1,$$
$$L^{s+k+1} = \{x_4, \dots, x_n, y_{s+k-1}, \dots, y_m\}.$$

Let us suppose $k = 1,$ then
$$L^{s+2} = \{x_4, \dots, x_n, y_{s}, \dots,
y_m\}.$$

Since $x_3 \notin L^{s+2}$ and the vector space $L^{s+1}$ is
generated by multiplying the space $L^s$ to elements $x_1$ and
$y_1$ on the right side (because of Leibniz superidentity), it
follows that element $x_3$ is obtained by the product $[x_2,x_1]$,
i.e. $[x_2, x_1] =ax_3 + \sum\limits_{i\geq 4} (*)x_i, \ a\neq0.$
Making the change of basic element as $x'_3 = ax_3 +
\sum\limits_{i\geq 4} (*)x_i,$ we can conclude that $[x_2, x_1]
=x_3.$

Let us define the products $[y_{s-j}, y_j], 1 \leq j \leq s-1.$

Applying the Leibniz superidentity and induction on $j$ we prove

$$[y_{s-j}, y_j] = (-1)^{j+1} \beta_{s-1,2}x_2 + \sum\limits_{i\geq 3}
(*)x_i. \eqno (3)$$

Indeed, for $j=1$ it is true by notation (2). Let us suppose that
equality (3) holds for $j=t.$ Then for $j = t+1$ we have
$$[y_{s-t-1}, y_{t+1}] = [y_{s-t-1}, [y_t, x_1]] =
[[y_{s-t-1}, y_t], x_1] - [[y_{s-t-1}, x_1],y_t] =$$ $$= -
[y_{s-t}, y_t] + [\sum\limits_{i\geq 2} (*)x_i, x_1] = -(-1)^{t+1}
\beta_{s-1,2}x_2 + \sum\limits_{i\geq 3} (*)x_i+$$
$$+[\sum\limits_{i\geq 2} (*)x_i, x_1]=(-1)^{t+2}
\beta_{s-1,2}x_2 + \sum\limits_{i\geq 3} (*)x_i.$$

It should be noted that in the decompositions of $[y_s, y_1]$ and
$[y_1, y_s]$ the coefficients at the basic elements $x_2$ and
$x_3$ are equal to zero.

Let us define the products $[y_{s+1-j}, y_j], 2 \leq j \leq s.$ In
fact, if $j=2,$ then

$$[y_{s-1}, y_2] = [y_{s-1}, [y_1, x_1]] = [[y_{s-1}, y_1], x_1] - [y_s, y_1] =$$ $$=[
\sum\limits_{i=2}^n\beta_{s-1,i}x_i, x_1] -
\sum\limits_{i=4}^n\beta_{s,i}x_i = \beta_{s-1,2}x_3 +
\sum\limits_{i\geq 4} (*)x_i,$$

Inductively applying the above arguments for  $j \geq 3$ and using
equality (3) we obtain

$$[y_{s+1-j}, y_j] = (-1)^j (j-1)
\beta_{s-1,2}x_3 + \sum\limits_{i\geq 4} (*)x_i, \quad 2 \leq j
\leq s. \eqno (4)$$

In particular, $[y_1, y_s] = (-1)^s (s-1) \beta_{s-1,2}x_3 +
\sum\limits_{i\geq 4} (*)x_i.$ On the other hand (as mentioned
above), in the decompositions of $[y_1, y_s]$ the coefficient at
the basic element $x_3$ is equal to zero. Therefore,
$\beta_{s-1,2}=0,$ which contradicts to the condition
$\beta_{s-1,2}\neq 0.$ Thus, our assumption $k=1$ is not possible.

Hence, $k \geq 2$ and we have
$$L^{s+2} = \{x_3, \dots, x_n, y_{s+1}, \dots,
y_m\}.$$

Since $y_s \notin L^{s+2},$ it follows that $$ \alpha_{2,s}\neq 0,
\quad \alpha_{2,j}=0 \quad \mbox{for}\quad  j<s,$$
$$\alpha_{i,j}=0, \quad \mbox{for any} \quad i \geq 3, \quad j<s+1.$$

Consider the product $$[[y_{s-1}, y_1], y_1] = \frac 1 2 [ y_{s-1},
[ y_1, y_1]] = \frac 1 2 [y_{s-1},
\sum\limits_{i=2}^n\beta_{1,i}x_i] .$$

The element $y_{s-1}$ belongs to $L^{s-1}$ and elements $x_2, x_3,
\dots, x_n$ lie in $L^3.$ Hence $\frac 1 2 [y_{s-1},
\sum\limits_{i=2}^n\beta_{1,i}x_i] \in L^{s+2}.$ Since $L^{s+2} =
\{x_3, \dots, x_n, y_{s+1}, \dots, y_m\},$ we obtain that
$[[y_{s-1}, y_1], y_1] =\sum\limits_{j\geq s+1}(*)y_j.$

On the other hand,
$$[[y_{s-1}, y_1], y_1] = [\sum\limits_{i=2}^n\beta_{s-1,i}x_i, y_1]
= \sum\limits_{i=2}^n\beta_{s-1,i}[x_i, y_1]= $$
$$=\sum\limits_{i=2}^n\beta_{s-1,i} \sum\limits_{j\geq s}\alpha_{i,j}y_j = \beta_{s-1,2}
\alpha_{2,s}y_s
 + \sum\limits_{j\geq s+1}(*)y_j.$$

Comparing the coefficients at the basic elements we obtain $
\beta_{s-1,2} \alpha_{2,s}=0,$ which contradicts the conditions
$\beta_{s-1,2}\neq 0$ and  $\alpha_{2,s} \neq 0.$

Thus, we get a contradiction with assumption that the superalgebra
$L$ has nilindex equal to $n+m$ and therefore the assertion of the
theorem holds.
\end{proof}

The investigation of the case $y_2 \in L^3$ shows that it depends
on the structure of the Leibniz algebra $L_0.$ So, we present some
remarks on naturally graded nilpotent Leibniz algebras.

Let $A = \{z_1, z_2, \dots, z_n\}$ be an $n$-dimensional nilpotent
Leibniz algebra of nilindex $p$ ($p<n$). Note that algebra $A$ is
not single-generated.

Consider the case where $gr(A)$ is a non-Lie Leibniz algebra.

\begin{lem}\label{l2}
Let $gr(A)$ be a naturally graded non-Lie Leibniz algebra. Then
$dim(gr(A)_1) + dim(gr(A)_2) \geq 4.$
\end{lem}
\begin{proof}
The construction of $gr (A)$ implies that every subspace $gr (A)
_i$ for $1\leq i \leq p-1$ is not empty. Obviously, $dim (gr
(A)_1) \geq 2 $ (otherwise $p=n+1$). If the dimension of subspace
$gr (A)_1$ is greater than two, then the statement of the lemma is
true. If $dim (gr (A) _1) = 2$ and $dim (gr (A) _2) = 2,$ then the
assertion of the lemma is evident.

Let us suppose that $dim (gr (A) _1) = 2$ and $dim (gr (A) _2) =
1.$ Then taking into account the condition $p<n$ we conclude that
there exists some $t \ (t > 2)$ such that $dim (gr (A)_t) \geq 2$
(otherwise the nilindex equals $n$).

Let $t_0 \ (t_0  > 2)$ be the smallest number with condition $dim
(gr (A)_{t_0}) \geq 2. $ Then
$$gr(A)_1 = \{\overline{z}_1, \overline{z}_2\}, \ gr(A)_2 =
\{\overline{z}_3\},\  \dots, \ gr(A)_{t_0-1} =
\{\overline{z}_{t_0}\}, \ gr(A)_{t_0}= \{\overline{z}_{{t_0}+1},
\overline{z}_{{t_0}+2}\}.$$

Using the argumentation similar to the one from \cite{Ver} we
obtain that
$$\left\{\begin{array}{ll} [\overline{z}_1, \overline{z}_1] =\alpha_1
\overline{z}_3,& \\ {}[\overline{z}_2, \overline{z}_1] = \alpha_2
\overline{z}_3,& \\ {}[\overline{z}_1, \overline{z}_2] = \alpha_3
\overline{z}_3,& \\ {}[\overline{z}_2, \overline{z}_2] = \alpha_4
\overline{z}_3,& \\ {}[\overline{z}_i, \overline{z}_1] =
\overline{z}_{i+1}, & 3 \leq i \leq t_0.\end{array}\right.$$

Since $gr(A)$ is non-Lie Leibniz algebra then there exists an
element of $gr(A)_1$ the square of which is non zero. It is not
difficult to see that $ \overline{z}_3, \dots,
\overline{z}_{{t_0}+1}$ belong to the right annihilator
$\Re(gr(A)),$ which is defined as
$$\Re(gr(A))= \{\overline{z} \in gr(A) | \ [\overline{y},\overline{z}]=0 \ \mbox{for any} \
\overline{y} \in gr(A)\}.$$

Moreover, one can assume $[\overline{z}_{t_0}, \overline{z}_2] =
\overline{z}_{{t_0}+2}.$

On the other hand, $$ [\overline{z}_{t_0}, \overline{z}_2] =
[[\overline{z}_{{t_0}-1}, \overline{z}_1], \overline{z}_2] =
[\overline{z}_{{t_0}-1},[ \overline{z}_1, \overline{z}_2]] +
[[\overline{z}_{{t_0}-1}, \overline{z}_2], \overline{z}_1] =$$
$$=
[\overline{z}_{{t_0}-1}, \alpha_3\overline{z}_3] + [
\beta\overline{z}_{t_0}, \overline{z}_1] =
\beta\overline{z}_{{t_0}+1}.$$

The obtained equality $\overline{z}_{{t_0}+2} =
\beta_{{t_0}-1,2}\overline{z}_{{t_0}+1}$ derives a contradiction,
which leads to the assertion of the lemma.
\end{proof}

From the Lemma \ref{l2} the corollary follows.
\begin{cor} \label{c1} Let $A$ be a Leibniz algebra satisfying the condition of
the Lemma \ref{l2}. Then $dim(A^3 ) \leq n-4.$ \end{cor}

The result on nilindex of the superalgebra under condition
$dim(L_0^3) \leq n-4$ is established in the following proposition.

\begin{prop}\label{p1} Let $L=L_0 \oplus L_1$ be a Leibniz superalgebra from $Leib_{n,m}$
with characteristic sequence $(n_1, \dots, n_k | m)$  and
$dim(L_0^3) \leq n-4.$ Then $L$ has a nilindex less than $n+m.$
\end{prop}
\begin{proof}
Let us assume the contrary, i.e. the nilindex of the superalgebra
$L$ is equal to $n+m.$ According to the Theorem \ref{t2} we need
to consider the case where $y_2$ belongs to $L^3,$ which leads to
$x_2 \notin L^3.$ Thus, we have $$ L^3 = \{x_3, x_4, \dots, x_n,
y_2, y_3, \dots, y_m\}.$$

From the condition $dim(L_0^3) \leq n-4$ it follows that there
exist at least two basic elements, that belong to $L_0^2 \setminus
L_0^3.$ Without loss of generality, one can assume $x_3, x_4 \in
L_0^2 \setminus L_0^3.$

Let $s$ be a natural number such that $x_3 \in L^{s+1}\setminus
L^{s+2},$
 $$L^{s+1} = \{x_3, x_4, \dots, x_n, y_s, y_{s+1}, \dots,
y_m\}, \ s \geq 2,$$
$$L^{s+2} = \{x_4, \dots, x_n, y_s, y_{s+1},
\dots, y_m\}.$$

The condition $x_4 \notin  L_0^3$ implies that $x_4$ can not be
obtained by the products $[x_i, x_1],$ with $3 \leq i\leq n.$
Therefore, it is generated by products $[y_j, y_1], s \leq j \leq
m.$ Hence, $L^{s+3} = \{x_4, \dots, x_n, y_{s+1}, \dots, y_m\}$
and $y_s \in L^{s+2} \setminus L^{s+3},$ which implies
$\alpha_{3,s} \neq 0.$

Consider the chain of equalities $$[[x_3, y_1], y_1] = [\sum
\limits_{j=s}^m \alpha_{3,j}y_j, y_1] =  \sum \limits_{j=s}^m
\alpha_{3,j}[y_j, y_1]= \alpha_{3,s}\beta_{s,4}x_4 + \sum
\limits_{i\geq 5}(*)x_i.$$

On the other hand,
$$[[x_3, y_1], y_1] = \frac 1 2 [x_3, [ y_1, y_1]] =  [x_3,
\sum \limits_{i=2}^n\beta_{1,i}x_i] =  \sum
\limits_{i=2}^n\beta_{1,i} [x_3, x_i] = \sum \limits_{i\geq
5}(*)x_i.$$

Comparing the coefficients at the corresponding basic elements in
these equations we get $\alpha_{3,s}\beta_{s,4} = 0,$ which
implies $ \beta_{s,4} = 0.$ Thus, we conclude that $x_4 \in
L^{s+4}$ and
$$L^{s+4} = \{x_4, \dots, x_n, y_{s+2}, \dots, y_m\}.$$

Let $k$ ($4 \leq k \leq m-s+2$) be a natural number such that $x_4
\in L^{s+k} \setminus L^{s+k+1}.$ Then for the powers of
descending lower sequences we have

$$L^{s+k-2} = \{x_4, \dots, x_n,
y_{s+k-4}, \dots, y_m\},$$
$$L^{s+k-1} = \{x_4, \dots, x_n,
y_{s+k-3}, \dots, y_m\},$$ $$L^{s+k} = \{x_4, \dots, x_n, y_{s+k-2},
\dots, y_m\},$$ $$L^{s+k+1} = \{x_5, \dots, x_n, y_{s+k-2}, \dots,
y_m\}.$$

It is easy to see that in the decomposition $[y_{s+k-3}, y_1] =
\sum\limits_{i=4}^n \beta_{s+k-3,i}x_i$ we have $\beta_{s+k-3,4}
\neq 0.$

Consider the equalities
$$[y_{s+k-4}, y_2] = [y_{s+k-4}, [y_1, x_1]] = [[y_{s+k-4},
y_1], x_1] - [[y_{s+k-4}, x_1], y_1] =$$ $$=[ \sum\limits_{i=3}^n
\beta_{s+k-3,i}x_i, x_1] - [y_{s+k-3}, y_1]= - \beta_{s+k-3,4}x_4 +
\sum\limits_{i\geq 5}(*)x_i.$$

Since $y_{s+k-4} \in L^{s+k-2}, y_2 \in L^3$ and $\beta_{s+k-3,4}
\neq 0,$ then the element $x_4$ should lie in $L^{s+k+1},$ but it
contradicts to $L^{s+k+1} = \{x_5, \dots, x_n, y_{s+k-2}, \dots,
y_m\}.$ Thus, the superalgebra $L$ has a nilindex less than $n+m.$
\end{proof}

From Proposition \ref{p1} we conclude that Leibniz superalgebra
$L= L_0 \oplus L_1$ with the characteristic sequence $(n_1, \dots,
n_k| m)$ and nilindex $n+m$ can appear only if $dim(L_0^3) \geq
n-3.$ Taking into account the condition $n_1 \leq n-2$ and
properties of naturally graded subspaces $gr(L_0)_1$ and
$gr(L_0)_2$ we get $dim(L_0^3) = n-3.$

Let $dim(L_0^3) =n-3.$ Then $$gr(L_0)_1 = \{\overline{x}_1,
\overline{x}_2\}, \ gr(L_0)_2 = \{\overline{x}_3\}.$$

Then, by Corollary \ref{c1} the naturally graded Leibniz algebra
$gr(L_0)$ is a Lie algebra, i.e. the following multiplication
rules are true
$$
\left\{ \begin{array}{l} [\overline{x}_1,\overline{x}_1]=0, \\{}
[\overline{x}_2,\overline{x}_1]=\overline{x}_3,   \\{}
 [\overline{x}_1,\overline{x}_2]=-\overline{x}_3, \\ {}[\overline{x}_2,\overline{x}_2]=0.
\end{array}\right.
$$

Using these products for the corresponding products in the Leibniz
algebra $L_0$ with the basis $\{x_1, x_2, \dots, x_n\}$ we have

$$
\left\{ \begin{array}{l} [x_1,x_1]=\gamma_{1,4}x_4 +
\gamma_{1,5}x_5 + \dots + \gamma_{1,n}x_n,
\\{} [x_2,x_1]=x_3, \\{} [x_1,x_2]=-x_3 + \gamma_{2,4}x_4 + \gamma_{2,5}x_5 + \dots
+ \gamma_{2,n}x_n, \\ {}[x_2,x_2]=\gamma_{3,4}x_4 +
\gamma_{3,5}x_5 + \dots + \gamma_{3,n}x_n.
\end{array} \right. \eqno(5)$$

\begin{prop}\label{p2}
Let $L=L_0\oplus L_1$ be a Leibniz superalgebra from $Leib_{n,m}$
with characteristic sequence $(n_1,  \dots, n_k | m)$ and
$dim(L_0^3)=n-3.$ Then $L$ has a nilindex less than $n+m.$
\end{prop}
\begin{proof}
Let us suppose the contrary, i.e. the nilindex of the superalgebra
$L$ equals $n+m.$ Then by Theorem \ref{t2} we can assume $x_2
\notin L^3.$ Hence,
$$ L^2 = \{x_2, x_3,\dots, x_n, y_2, y_3, \dots, y_m\},$$
$$L^3 = \{x_3, x_4, \dots, x_n, y_2, y_3, \dots, y_m\}.$$

If $y_2 \in L^4,$ then it should be generated from the products
$[x_i, y_1], 3 \leq i \leq n,$ but elements $x_i, (3 \leq i \leq
n)$ are in $L_0^2.$ Therefore, they are generated by linear
combinations of products of elements from $L_0.$ The equalities
$$[[x_i, x_j], y_1] =  [x_i,[ x_j, y_1]] + [[x_i, y_1], x_j] =
[x_i,\sum\limits_{t=2}^m \alpha_{j,t}y_t] + [\sum\limits_{t=2}^m
\alpha_{i,t}y_t, x_i] = \sum\limits_{t\geq 3} (*)y_t$$ show that
the element $y_2$ can not be obtained by the products $[x_i, y_1],
3 \leq i \leq n,$ i.e. $y_2\notin L^4.$ Thus, we have
$$L^4 = \{x_3, x_4, \dots, x_n, y_3, \dots, y_m\}.$$

The simple analysis of descending lower sequences $L^3$ and $L^4$
derives $\alpha_{2,2} \neq 0.$

Let $s$ be a natural number such that $x_3 \in L^{s+1}\setminus
L^{s+2},$ i.e.

$$L^s = \{x_3, x_4, \dots, x_n, y_{s-1}, y_s, \dots, y_m\}, s
\geq 3,$$ $$L^{s+1} = \{x_3, x_4, \dots, x_n, y_s, y_{s+1}, \dots,
y_m\},$$ $$L^{s+2} = \{x_4, \dots, x_n, y_s, y_{s+1}, \dots, y_m\}
\ \ \mbox{and} \ \ \beta_{s-1,3} \neq 0.$$

If $s=3,$ then $\beta_{2,3} \neq 0$ and we consider the product
$$[[x_2, y_1], y_1] = [ \sum\limits_{j=2}^m
\alpha_{2,j}y_j, y_1] =  \sum\limits_{j=2}^m \alpha_{2,j}[y_j, y_1]
= \alpha_{2,2} \beta_{2,3}x_3 + \sum\limits_{i\geq 4}(*)x_4 .$$

 On the other hand, $$[[x_2, y_1], y_1] = \frac 1 2 [x_2, [ y_1, y_1]] =
\frac 1 2 [x_2, \sum\limits_{i=2}^n\beta_{1,i}x_i] =
\sum\limits_{i\geq 4}(*)x_i.$$

Comparing the coefficients at the corresponding basic elements we
get equality $\alpha_{2,2}\beta_{2,3} = 0,$ i.e. we have a
contradiction with supposition $s=3.$

If $s \geq 4,$ then consider the chain of equalities
$$[y_{s-2}, y_2] = [y_{s-2}, [y_1, x_1]] = [[y_{s-2}, y_1], x_1] -
[[y_{s-2}, x_1], y_1] =$$ $$= [
\sum\limits_{i=3}^n\beta_{s-2,i}x_i , x_1] - [y_{s-1}, y_1] = -
\beta_{s-1,3}x_3 + \sum\limits_{i\geq 4}(*)x_i.$$

Since $y_{s-2} \in L^{s-1}$ and $y_2 \in L^3$ then $x_3 \in
L^{s+2} = \{x_4, \dots, x_n, y_{s-1}, \dots, y_m\},$ which is a
contradiction with the assumption that the nilindex of $L$ is
equal to $n+m.$
\end{proof}

Summarizing the Theorem \ref{t2} and Propositions \ref{p1},
\ref{p2} we have the following result
\begin{thm}\label{t3}
Let $L=L_0 \oplus L_1$  be a Leibniz superalgebra from
$Leib_{n,m}$ with characteristic sequence equal to $(n_1, \dots,
n_k | m).$ Then the nilindex of the Leibniz superalgebra $L$ is
less than $n+m.$
\end{thm}

\end{document}